\documentclass[a4paper,12pt]{amsart}
\usepackage[utf8]{inputenc}

\usepackage{url}
\usepackage[margin=1in]{geometry}  
\usepackage{epsfig}
\usepackage{color}
\usepackage{graphicx}              
\usepackage{amsmath}
\usepackage{amscd}               
\usepackage{amsfonts}
\usepackage{amssymb}             
\usepackage{amsthm}                
\usepackage{epsfig}
\usepackage{mathrsfs}
\usepackage{hyperref}
\usepackage{enumerate}

\newtheorem{theorem}{Theorem}

\newtheorem{corollary}[theorem]{Corollary}

\newtheorem*{theorem*}{Theorem}
\newtheorem{lemma}[theorem]{Lemma}
\newtheorem*{lemma*}{Lemma}
\newtheorem{proposition}[theorem]{Proposition}
\newtheorem*{proposition*}{Proposition}
\theoremstyle{definition}

\theoremstyle{remark}

\newtheorem*{remark}{Remark}
\newtheorem*{rem}{Remark}

\newcommand{\good}{\operatorname{good}}
\newcommand{\LCM}{\operatorname{LCM}}

\newcommand{\GCD}{\operatorname{GCD}}
\newcommand{\sqf}{\operatorname{sqf}}
\newcommand{\diag}{\operatorname{diag}}
\newcommand{\op}{\operatorname{op}}
\newcommand{\fix}{\operatorname{fix}}
\newcommand{\Sym}{\operatorname{Sym}}

\newcommand{\E}{\mathbf{E}}     
\newcommand{\Prob}{\mathbf{P}} 

\newcommand{\bR}{\mathbb{R}}      

\newcommand{\zed}{\mathbb{Z}}

\newcommand{\bN}{\mathbb{N}}

\newcommand{\ba}{\mathbf{a}}

\newcommand{\cB}{\mathscr{B}}

\newcommand{\cN}{\mathscr{N}}	
\newcommand{\cK}{\mathscr{K}}    
\newcommand{\cP}{\mathscr{P}}		 
\newcommand{\cM}{\mathscr{M}}
\newcommand{\cC}{\mathscr{C}}

\newcommand{\cD}{\mathscr{D}}
\newcommand{\cOD}{\mathscr{OD}}

\newcommand{\cS}{\mathscr{S}}
\newcommand{\cX}{\mathscr{X}}

\newcommand{\vv}{{\mathbf{v}}}
\newcommand{\vtheta}{{\boldsymbol{\theta}}}

\newcommand{\ux}{\underline{x}}

\newcommand{\uG}{\underline{G}}
\newcommand{\uO}{\underline{0}}
\newcommand{\ueps}{\underline{\epsilon}}
\newcommand{\umu}{\underline{\mu}}
\newcommand{\upi}{\underline{\pi}}

\newcommand{\utau}{\underline{\tau}}

 \newcommand{\constPiGood}{0.3}    
\newcommand{\constRhoTwo}{0.9}
\newcommand{\constRhoOp}{0.1}
 \newcommand{\constEBTwo}{0.246514091}
\newcommand{\constEBop}{0.002220166}

\newcommand{\constM}{1.769746269}
\newcommand{\constMBTwoTwo}{0.391292208}
\newcommand{\constMBTwo}{0.625533539}
\newcommand{\constFOa}{1.769746269}
\newcommand{\constFOb}{1.900670975}
\newcommand{\constFOc}{2.033321919}
\newcommand{\constFOd}{2.184489901}
\newcommand{\constFOe}{2.363269323}
\newcommand{\constFOf}{2.530235874}
\newcommand{\constFOg}{2.686345986}
\newcommand{\constFOh}{2.833661687}
\newcommand{\constFOi}{2.973253326}
\newcommand{\constFOj}{3.106051540}
\newcommand{\constbetaTwonew}{94.66051416}
\newcommand{\constbetaThreenew}{199.2834489}
\newcommand{\constSplice}{100}
\newcommand{\constEpsMax}{ 0.292129153}
\newcommand{\constOmegaThreshold}{5}

\newcommand{\constbetaTwoNewratio}{0.5197033883}
\newcommand{\constbetaThreeNewratio}{0.3100980448}
\newcommand{\constMIter}{2.949873427}
\newcommand{\constcCTwo}{0.0001571422884}
\newcommand{\constcSTwo}{3.212501212}
\newcommand{\constrhoThreeIter}{0.7}
\newcommand{\constEBThreeThreeIter}{0.1023637064}
\newcommand{\constrhoTwoIter}{0.2}
\newcommand{\constEBTwoTwoIter}{0.6144485964}
\newcommand{\constRhoOpIter}{0.1}
\newcommand{\constEBopTwoIter}{0.0005048197920}
\newcommand{\constMBThreeThreeIter}{0.2089055233}
\newcommand{\constMBTwoTwoIter}{4.388918546}

\newcommand{\constMBThreeIter}{0.5933577790}
\newcommand{\constMaxEpsIter}{0.190000303}
\newcommand{\constFOaIter}{1.459164221}
\newcommand{\constFObIter}{1.780349459}
\newcommand{\constFOcIter}{2.096937862}
\newcommand{\constFOdIter}{2.387653719}
\newcommand{\constFOeIter}{2.656941273}
\newcommand{\constFOfIter}{2.909180305}
\newcommand{\constFOgIter}{3.147611526}
\newcommand{\constFOhIter}{3.374605257}
\newcommand{\constFOiIter}{3.591932780}
\newcommand{\constFOjIter}{3.800951606}
\newcommand{\constbetaTwoGrowthIter}{48.515}
\newcommand{\constbetaThreeGrowthIter}{487.17}
\newcommand{\constloglengthiter}{1.5}

\newcommand{\constProdFactor}{1.004212}
\newcommand{\constProdBound}{1.506318}
\newcommand{\constSumTwoFactor}{1.002631}
\newcommand{\constSumThreeFactor}{1.004382}
\newcommand{\constTauTwoOffset}{0.00334}
\newcommand{\constTauThreeOffset}{0.00779}
\newcommand{\constTauTwoBound}{1.21974}
\newcommand{\constTauThreeBound}{2.84605}

\title{Covering systems with restricted divisibility}
\author{Robert D. Hough and Pace P. Nielsen}
\thanks{The project was sponsored by the National Security Agency under Grant
Number H98230-16-1-0048 and by the National Science Foundation under grant numbers DMS-1712682 and DMS-1802336.}
\begin{document}

\begin{abstract}
We prove that every distinct covering system has a
modulus divisible by either 2 or 3.
\end{abstract}

\maketitle

\section{Introduction}
A covering system of congruences is a collection
\[
 a_i \bmod m_i, \qquad i = 1, 2, ..., k
\]
such that every integer satisfies at least one of them.  A covering system is
distinct if the moduli $m_i$ are distinct and greater than 1. Erd\H{o}s introduced the idea of a distinct covering system of congruences in constructing an arithmetic progression of odd numbers, none of whose members are prime \cite{E50}. In the paper \cite{E50} Erd\H{o}s asked whether the least modulus of a distinct covering system of congruences can be arbitrarily large.  The first author recently answered this question in the negative \cite{H15}, proving that the least modulus of a distinct covering system of congruences is at most $10^{16}$.  The largest known minimum modulus is 42, given by Tyler Owens \cite{O14}.  A second old problem of
Erd\H{o}s and Selfridge asks whether there exists a distinct covering system of
congruences with all moduli odd. According to \cite{FFK00} Erd\H{o}s has offered \$25 for the proof that no odd distinct covering system of congruences exists, while Selfridge has offered \$2000 for a construction of an odd distinct covering system. Schinzel proved that a negative answer to the odd modulus problem has applications to the irreducibility of families of polynomials. While the odd modulus problem remains open,  Simpson and Zeilberger \cite{SZ91} proved that a distinct covering system consisting of odd square-free numbers involves at least 18 primes, which was improved to 22 primes by Guo and Sun \cite{GS05}.  This paper makes further negative progress towards the odd modulus problem.
\begin{theorem}\label{main_theorem}
 Every distinct covering system of congruences has a modulus divisible by
either $2$ or $3$.
\end{theorem}
\noindent
This answers a problem raised in \cite{G13}.
\section{Set-up}
Suppose given a finite set of moduli $\cM$, and, for each $m \in \cM$, a set of
residues $\ba_m$ modulo $m$.  Let
\[
Q = \LCM(m: m \in \cM)
\]
and
\[
 R = \zed \setminus \bigcup_{m \in \cM} (\ba_m \bmod m),
\]
which is a set defined modulo $Q$.  One
way to show that the congruences
\[(\ba_m \bmod m), \qquad m \in \cM\] do not cover the integers is to give a
positive lower bound for the density of $R$. The proof of Theorem
\ref{main_theorem} gives such a lower bound, although quantitatively it
estimates some related quantities.

If we let
$\zed/Q\zed$ have the uniform probability measure, then the density of $R$ is
equal
to its probability.  For $m\in \cM$ let $A_m$ be the event $(\ba_m \bmod m)$,
which has probability
$\frac{|\ba_m|}{m}$, and extend this to $m |Q$ with $m \not \in \cM$ by
setting $A_m = \emptyset$ for these $m$.  Then
\begin{equation}\label{R_density}
 \Prob(R) = \Prob\left(\bigcap_{ m |Q} A_m^c\right).
\end{equation}
A familiar argument (the Chinese
Remainder Theorem) implies that $A_m$ is independent of any set of congruences
to moduli co-prime to $m$.  Thus a valid dependency graph for the events
$\{A_m: m|Q\}$ has edge $(m,m')$ if and only if $\GCD(m,m')>1$.

A family of results connected to the Lov\'{a}sz Local Lemma give
worst-case lower bounds for the probability of an intersection as in
(\ref{R_density}), taking as input only the events' probabilities and their
dependency graph.  In principle we could hope to prove Theorem
\ref{main_theorem} by directly applying one of these results to claim that the
uncovered set always has a non-zero density, but, as we will see, such a lower
bound cannot
be given, and further input is needed.  Two methods of Lov\'{a}sz type do
figure into our argument, however, as we will describe.

Given the problem of estimating from below the probability of the
intersection of the complements of some events given only their probabilities
and their dependency graph, the best possible estimate has been given by
Shearer \cite{S85}.  The estimate is best possible in the sense that the
argument constructs a probability space and events having the prescribed
probabilities and dependency graph, and such that the lower bound holds with
equality.  However, the condition with which Shearer's result holds can be
difficult to verify, and so the following result is useful because it is
easy to check. Note that this is essentially due to
\cite{SZ91} in
this context.

\begin{theorem}[Shearer-type theorem]\label{initial_shearer_theorem}
 Suppose we have a probability space.  Let $[n] = \{1, 2, ..., n\}$, and assume
that for each $1 \leq i \leq n$ there is a weight $\pi_i$ assigned,
satisfying $1 \geq
\pi_1 \geq \pi_2 \geq ... \geq \pi_n \geq 0$.  Let the sets $\emptyset
\neq T \subset [n]$ index events $A_T$ each having probability
\[
 0 \leq \Prob(A_T) \leq \prod_{t \in T} \pi_t:= \pi_T.
\]
Assume that $A_T$ is independent of the $\sigma$-algebra generated by $\{A_S: S \subset [n], S\cap T
= \emptyset\}$, so that a valid dependency graph for the events $\{A_T:
\emptyset \neq T \subset[n]\}$ has an edge between $S \neq T$ whenever $S \cap
T \neq \emptyset$.

Define $\rho(\emptyset) =1$, and given $\emptyset \neq T \subset [n]$, set
(put an arbitrary total ordering $<$ on $2^{[n]}$ to avoid confusion)
\[
 \rho(T) = 1 - \sum_{\emptyset \neq S_1 \subset T} \pi_{S_1} + \sum_{\substack{
\emptyset \neq S_1, S_2 \subset T\\ S_1<S_2 \text{ disjoint}}}
\pi_{S_1}\pi_{S_2}- \sum_{\substack{ \emptyset \neq S_1, S_2, S_3 \subset
T\\ S_1 < S_2 < S_3 \text{ disjoint}}} \pi_{S_1}\pi_{S_2}\pi_{S_3} + ....
\]
Suppose that
$
 \rho([1]) \geq \rho([2]) \geq ... \geq \rho([n]) > 0.
$
Then for any $\emptyset \neq T \subset [n]$,
\begin{equation}\label{shearer_bound}
 \Prob\left(\bigcap_{\emptyset \neq S \subset T} A_S^c\right) \geq \rho(T) > 0
\end{equation}
and, for any $ T_1 \subset T_2 \subset [n]$,
\begin{equation}\label{relative_shearer_bound}
 \frac{\Prob\left(\bigcap_{\emptyset \neq S \subset T_2}
A_S^c\right)}{\Prob\left(\bigcap_{\emptyset \neq S \subset T_1} A_S^c\right)}
\geq \frac{\rho(T_2)}{\rho(T_1)}.
\end{equation}

\end{theorem}

We prove a slightly more general version of this theorem in Appendix
\ref{shearer_appendix}.

To apply the Shearer-type theorem  in the context of Theorem
\ref{main_theorem}, order the primes greater than
3 as $p_1 = 5, p_2 = 7, p_3 = 11, ...$.  Suppose we are given a distinct
congruence system with moduli formed with the primes $p_1, ..., p_n$.  Identify
$S \subset [n]$ with the square-free number $m_S = \prod_{i \in S} p_i$ and form
the event $A_S$ which is the union of all congruences having square-free part
$m_S$,
\[
 A_S = \bigcup_{m : \;\sqf(m) = m_S} (a_m \bmod m),
\]
where $\sqf(m) = \prod_{p: p|m} p$.  Then $A_S$ is an event with probability
\[
 \Prob(A_S) < \prod_{i \in S} \frac{1}{p_i - 1}.
\]
In particular, we may appeal to Theorem \ref{initial_shearer_theorem} with
$\pi_i =
\frac{1}{p_i-1}$.  Arguing in this way, we may check that there is no covering
composed of only the primes between 5 and 631, but at this point, the Shearer
function becomes negative, and no further result can be drawn from that
estimate.

What allows us to make further
progress is that, within the range in which
Shearer's theorem holds,  estimate
(\ref{relative_shearer_bound}) of Theorem \ref{initial_shearer_theorem} gives
substantial information about the structure of the uncovered set.  To see this,
suppose that we have a congruence system as above with uncovered set $R$, and
that Theorem \ref{initial_shearer_theorem} applies. We can estimate the
proportion
of the set $R$ that lies in a given congruence class $(b \bmod m)$ for $m
|Q$ by
\begin{align*}
 \frac{\Prob((b \bmod m) \cap R) }{\Prob(R)}  &= \frac{\Prob\left((b\bmod m)
\cap \bigcap_{m' \in \cM, m'|Q}(a_{m'} \bmod m')^c\right)}{\Prob
\left(\bigcap_{m' \in \cM, m' | Q} (a_{m'} \bmod m')^c \right)}\\& \leq
\Prob((b \bmod m)) \frac{\Prob\left(\bigcap_{m' \in \cM, m'|Q,
(m,m')=1}(a_{m'} \bmod m')^c \right)}{\Prob
\left(\bigcap_{m' \in \cM, m' | Q} (a_{m'} \bmod m')^c \right)}
\\& = \frac{1}{m} \frac{\Prob\left(\bigcap_{m' \in \cM, m'|Q,
(m,m')=1}(a_{m'} \bmod m')^c \right)}{\Prob
\left(\bigcap_{m' \in \cM, m' | Q} (a_{m'} \bmod m')^c \right)}.
\end{align*}
The ratio of probabilities on the right is bounded by the relative conclusion
(\ref{relative_shearer_bound}) of Theorem \ref{initial_shearer_theorem},
which gives a ratio of
$\frac{\rho([n]\setminus S_m)}{\rho([n])}$ where $[n]$ again represents the
full set of primes
dividing $Q$, and $S_m$ is those primes from $[n]$ which  divide $m$.
Thus
\begin{equation*}
 \frac{\Prob((b \bmod m) \cap R) }{\Prob(R)}  \leq \frac{1}{m}
\frac{\rho([n]\setminus S_m)}{\rho([n])}.
\end{equation*}
If $S_m$ is such that $\rho([n]\setminus S_m) \approx \rho([n])$ then we deduce
that $R$ is almost uniformly distributed across residues modulo $m$.

We summarize the  above discussion in the following Theorem.
\begin{theorem}\label{shearer_applied_theorem}
 Let $p_1 < p_2 < ...< p_n$ be a sequence of primes, and let weights
$\pi_1, ..., \pi_n$ given by $\pi_i = \frac{1}{p_i-1}$.   For a
subset $S \subset [n]$ identify $S$ with $q_S = \prod_{p \in S}p,$ and write
$\rho(q) =
\rho(q_S) = \rho(S)$ for the Shearer function associated to $S$ with
weights $\pi_i$, as in Theorem \ref{initial_shearer_theorem}.

Suppose that $\rho(p_1) \geq \rho(p_1p_2) \geq
... \geq \rho(p_1p_2...p_n) > 0$.  Then any distinct congruence system with
moduli
composed only of $p_1, ..., p_n$ does not cover the integers.  Moreover, if $R$
is the uncovered set and if $m$ is a modulus composed of
primes corresponding to a set $S \subset [n]$ then
\begin{equation}\label{shearer_progression}
 \max_{b\bmod m} \frac{|R \cap (b \bmod m)|}{|R|} \leq
\frac{1}{m} \frac{\rho(q_{[n]\setminus S})}{\rho(q_{[n]})}.
\end{equation}

\end{theorem}

Although the sieving problem described in Theorem \ref{main_theorem} concerns
systems of congruences in which each congruence set $\ba_m$ has size 0 or 1, in
the course of our argument we consider congruences with  sets $\ba_m$ of
variable size.  In this situation the condition of Theorem
\ref{initial_shearer_theorem} becomes unwieldy and we appeal instead to the
following Theorem,
which follows from an improved form of the Lov\'{a}sz Local Lemma due to
\cite{BFPS11}, see also \cite{ScSo05}.

\begin{theorem}\label{LLL_type_theorem}
 Let $\cN \subset \bN_{>1}$ be a finite collection of moduli whose prime
factors are drawn from a set of primes $\cP$.  Let $\LCM(n: n\in
\cN) = Q$. Suppose that for each $n \in \cN$ a collection of residues $\ba_n
\bmod
n$
is given. Write
\[
 R = \zed \setminus \bigcup_{n \in \cN} (\ba_n \bmod n).
\]  Suppose that there
exist weights
$
 \{x_p\}_{p \in \cP}$ with $ x_p \geq 0
$,
which satisfy the constraints
\[
 \forall p \in \cP, \qquad x_p \geq \sum_{n \in \cN: p |n} \frac{|\ba_n \bmod
n|\prod_{p' |n } (1 + x_{p'})}{n}.
\]
Then the density of $R$ is at least
\begin{equation}\label{LLL_density}
\frac{|R \bmod Q|}{Q} \geq \exp\left( - \sum_{n \in \cN} \frac{|\ba_n \bmod
n|\prod_{p|n}(1 + x_p)}{n} \right)>0.
\end{equation}
 Also, for any $n \in \cN$,
\begin{equation}\label{relative_LLL}
 \max_{b \bmod n} \frac{|R \cap (b \bmod n) \bmod Q|}{|R \bmod Q|} \leq
\frac{\exp\left(\sum_{p|n} x_p\right)}{n}.
\end{equation}

\end{theorem}

\begin{remark} Conclusion (\ref{LLL_density}) corresponds to
(\ref{shearer_bound}) of Theorem
\ref{initial_shearer_theorem}, and (\ref{relative_LLL}) corresponds to
(\ref{shearer_progression}).
\end{remark}

If we write $\ux$ for $\{x_p\}_{p \in \cP}$ and $\uG(\ux)$ for
\[
 G_p(\ux) = \sum_{n \in \cN: p|n} \frac{|\ba_n \bmod n| \prod_{p'|n} (1 +
x_{p'})}{n}
\]
then the condition of Theorem \ref{LLL_type_theorem} equivalently asks for a
non-negative ($\ux \geq \uO$) fixed point $\uG(\ux) = \ux$, which is
relatively
easy to determine.  Thus, although Theorem \ref{LLL_type_theorem} is
strictly weaker than Theorem \ref{initial_shearer_theorem}, it is useful since
it is
more easily applied.

A proof and further discussion of  Theorem \ref{LLL_type_theorem} is given in
Section \ref{LLL_section}.

\section{Overview of  argument}

We now give an overview of our argument.  As the structure  is
similar to that of
 the minimum modulus problem we refer the proofs of some background statements
to \cite{H15}.

We assume given a congruence system with
finite set of moduli \[\cM \subset \{m > 1, (m,6)=1\},\] together with a
residue class $a_m \bmod m$ for each $m \in \cM$. We let
\[
 Q = \LCM(m: m \in \cM),
\]
and set
\[
 R = \zed \setminus \bigcup_{m \in \cM} (a_m \bmod m)
\]
for the set left uncovered by the congruence system.  Theorem
\ref{main_theorem}  follows by showing that the density of $R$ is positive.

To estimate the density of $R$ we appeal to  Lov\'{a}sz Local Lemma-type
arguments of the previous section.
These arguments, however, only apply to estimate the density of sets left
uncovered by congruence systems whose moduli are composed of a limited number of
primes, and so    we break the
estimate for the density of $R$ into stages.

Let $P_0 = 4 < P_1 < P_2 < ...$ be a sequence of
real numbers (not equal to prime integers). Let $Q_0 = 1$ and, for
$i \geq 1$, \[Q_i = \prod_{p^j \| Q, p < P_i} p^j\] be the part of $Q$ composed
of primes less than $P_i$.  We let $\cM_i = \{m \in \cM: m|Q_i\}$ be the
$P_i$-smooth moduli in $\cM$, and we let the set of `new
factors' be \[\cN_i = \{n > 1: n|Q_i, p|n \Rightarrow P_{i-1} < p \leq P_i\}.\]
Notice that each $m \in \cM_{i+1} \setminus \cM_{i}$ has a unique
factorization as $m = m_0 n$ with $m_0 |Q_{i}$ and $n \in \cN_{i+1}$.

We consider the sequence of sets $\zed = R_0
\supset R_1 \supset ...$,
\[
 \forall i \geq 1, \qquad R_i = \zed \setminus \bigcup_{m \in \cM_i} (a_m \bmod
m) .
\]
Since $R_i = R$ eventually, it will suffice to show that $R_i$ is non-empty for
each $i$.

The set $R_i$ is defined modulo $Q_i$.  Viewing $\zed/Q_{i+1}\zed$ as fibered
over $\zed/Q_i\zed$ we note that
\[
 R_{i+1} = R_{i}\setminus \bigcup_{m \in \cM_{i+1} \setminus \cM_{i}} (a_m
\bmod m),
\]
so that we may view $R_{i+1}$ as cut out from the fibers $(r \bmod Q_i)$, $r
\in R_i$, by congruences to moduli in $\cM_{i+1} \setminus \cM_i$. Given $r \in
R_i$
and $m \in \cM_{i+1} \setminus \cM_{i}$, factor $m = m_0 n$ with $m_0|Q_{i}$
and $n \in \cN_{i+1}$. Then the congruence $(a_m \bmod m)$ meets $(r \bmod
Q_{i})$
if and only if $r \equiv a_{m_0 n} \bmod m_0$, and when  it does so, it
intersects in a single residue class modulo $nQ_{i}$.  Thus, grouping together
moduli according to common new factor $n \in \cN_{i+1}$ we find
\[
 R_{i+1} \cap (r \bmod Q_i) = (r \bmod Q_i) \setminus \bigcup_{n \in \cN_{i+1}}
A_{n,r},
\]
with
\[
 A_{n,r} = (r \bmod Q_i) \cap \bigcup_{\substack{m_0 |Q_i\\ m_0 n \in
\cM_{i+1}}} (a_{m_0n} \bmod m_0n).
\]
After translating and dilating $(r \bmod Q_i)$ to coincide with the integers,
the set $A_{n,r}$ is composed of some residue classes modulo $n$, a set
which we call $\ba_{n,r}$. Thus we can
understand the problem of estimating the density of $R_{i+1}$ within $(r \bmod
Q_i)$ as sieving the integers by multiple residue classes to moduli in
$\cN_{i+1}$, a set of moduli whose prime factors are constrained to lie in
$[P_i, P_{i+1})$.  This is the situation treated by the Lov\'{a}sz-type
Theorem, Theorem \ref{LLL_type_theorem} above, and so, if we are able to
solve the relevant fixed-point problem then we obtain that the fiber is
non-empty.    Note that in the
initial stage, all of the sieving sets have size 0 or 1, so that in this stage
we can appeal to the optimal Shearer-type Theorem, Theorem
\ref{initial_shearer_theorem}.

In practice we will not estimate the density of $R_{i+1}$ over all of $R_i$,
but only within certain `good' fibers above a subset $R_i^* \subset R_i \bmod
Q_i$. We will be deliberately vague at this point about the requirements of a
good fiber.  Roughly these ensure that the corresponding fixed-point problem has
a
favorable solution.   Also, we
require that  $R_i^* \subset R_{i-1}^* \cap R_i$ so
that the good sets are nested.  We let $R_0^* = R_0 = \zed$.

For $i \geq 1$ we weight the set $\zed/Q_i\zed$ with a
probability
measure $\mu_i$ supported on $R_{i-1}^* \cap R_i$, chosen so as to guarantee
that a large proportion of the fibers are good. The measure $\mu_1$ is
uniform on the set $R_0^* \cap R_1 = R_1\subset \zed/Q_1\zed$,
\[
 \forall r \in R_0^* \cap R_1\bmod Q_1, \qquad \mu_1(r) = \frac{1}{|R_1 \bmod
Q_1|}.
\]
Taking the measure $\mu_i$ as given, define, for $i \geq 1$,
\[
 \pi_{\good}(i) = \frac{\mu_i(R_i^*)}{\mu_i(R_{i-1}^* \cap R_i)}
\]
to be the proportion of good fibers.
For $i \geq 1$ and $r \in R_{i}^* \cap R_{i+1}\bmod Q_{i+1}$ we set
\[
 \mu_{i+1}(r) = \frac{\mu_{i}(r \bmod Q_{i})}{\pi_{\good}(i)|R_{i+1} \cap (r
\bmod Q_{i}) \bmod
Q_{i+1}|}.
\]
Thus, for a fixed $r \in R_{i}$, $\mu_{i+1}$ is constant on $R_{i+1} \cap (r
\bmod
Q_{i})$. That $\mu_i$ is a sequence of probability measures follows from
\cite{H15} Lemma 2, although, note that the factor of
$\frac{1}{\pi_{\good}(i)}$ is not included in the definition of $\mu_i$ in
\cite{H15}, so that  the measures there do not have mass 1.
Throughout, when we write $\E_{r \in R_{i-1}^*\cap R_i}$ we mean expectation
with respect to the measure $\mu_i$.

Along with the measure $\mu_i$ we track some bias statistics of $R_{i-1}^*
\cap R_i$.  Let $\ell_k(m)$ be the multiplicative function given at primes
powers by
\[
 \ell_k(p^j) = (j+1)^k - j^k.
\]
For $i \geq 1$, the $k$th bias statistic of $R_{i-1}^* \cap R_i$ is defined to
be
\[
 \beta_k^k(i) = \sum_{m|Q_i} \ell_k(m) \max_{b \bmod m}\mu_i((b \bmod m)).
\]
The importance of the bias statistics is that they control moments of
(mixtures of) the sizes of the sets $\ba_{n,r}$ as $r$ varies in $R_{i-1}^* \cap
R_i$.
\begin{lemma}\label{convexity_lemma}
 Let $i \geq 1$.  Let $\{w_n: n \in \cN_{i+1}\}$ be any collection of
non-negative weights, not all of which are zero.  For
each $k \geq 1$ we have
 \[
  \E_{r \in R_{i-1}^* \cap R_i}
 \left(\sum_{n \in \cN_{i+1}} w_n |\ba_{n,r} \bmod n|\right)^k \leq
\left(\sum_{n \in \cN_{i+1}} w_n\right)^k \beta_k^k(i).
 \]
\end{lemma}
\begin{proof}
 See Lemmas 4 and 5 of \cite{H15}.
\end{proof}

In addition to the bias statistics, it will be useful for us to track maximum
biases among the various good fibers.  Let $i \geq 0$ and let $n \in
\cN_{i+1}$.  We define the maximum bias at $n$ to be
\[
b_n =  \max_{r \in R_{i}^*} \max_{b \bmod n} \frac{n |R_{i+1}\cap (r
\bmod Q_i) \cap (b\bmod n) \bmod Q_{i+1}|}{|R_{i+1} \cap (r \bmod Q_i) \bmod
Q_{i+1}|}.
\]
Note that these appeared only implicitly in \cite{H15}, but to
get a better quantitative bound it will be useful for us to track them more
carefully here.

The iterative growth of the bias statistics $\beta_k(i)$ to $\beta_k(i+1)$ is
controlled by the proportion of good fibers $\pi_{\good}(i)$ and the maximal
biases at $n \in \cN_{i+1}$.
\begin{lemma}\label{bias_stat_growth_lemma}
 Let $i \geq 1$.  For each $k \geq 1$ we have the bound
 \[
  \beta_k^k(i+1) \leq \frac{\beta_k^k(i)}{\pi_{\good}(i)} \left(1 + \sum_{n \in
\cN_{i+1}} \frac{\ell_k(n) b_n}{n}\right).
 \]
\end{lemma}
\begin{proof}
 This follows by tracing the proof of Proposition 3 of \cite{H15}.
\end{proof}

We now turn to giving a detailed account of Theorem \ref{LLL_type_theorem}.

\section{The Local Lemma and good fibers}\label{LLL_section}
Our Theorem \ref{LLL_type_theorem}, which is used to estimate the density of
good fibers, is derived from the following improved version of the Lov\'{a}sz
Local Lemma due to \cite{BFPS11}, see also \cite{ScSo05}.

\begin{theorem}[Clique Lov\'{a}sz Local Lemma]\label{sharp_lovasz_theorem}
Suppose that $G = (V,E)$ is a dependency graph for family of events
$\{A_v\}_{v \in V}$, each with probability $\Prob(A_v) \leq \pi_v$.  Let $N_v$
be
the neighborhood of $v \in V$. Suppose that
there exists sequence $\umu = \{\mu_v\}_{v \in V}$ of reals in $[0,\infty)$ such
that, for each $v \in V$,
\begin{equation}\label{lovasz_condition}
 \mu_v \geq \pi_v \phi_v(\umu)
\end{equation}
where
\[
 \phi_v(\umu) = \sum_{\substack{ R \subset \{v\} \cup N_v\\ R \text{ indep.
in } G}} \prod_{v' \in R} \mu_{v'}.
\]

Then
\begin{equation}\label{typical_conclusion}
 \Prob\left(\bigcap_{v \in V} A_v^c\right) \geq \exp\left(-\sum_{v \in
V} \mu_v\right)
\end{equation}
and, for all $U \subset V$,
\begin{equation}\label{relative_version}
 \frac{\Prob\left(\bigcap_{v \in V} A_v^c\right)}{\Prob\left(\bigcap_{u \in U}
A_u^c\right)} \geq \exp\left(-\sum_{v \in V \setminus U} \mu_v \right).
\end{equation}
\end{theorem}
\begin{remark}
 In the definition of $\phi_v$, $R = \emptyset$ is to be included, with
associated product equal to 1.
\end{remark}
\begin{proof}
This theorem with conclusion
\begin{equation}\label{typical_conclusion_paper}
 \Prob\left(\bigcap_{v \in V} A_v^c\right) \geq \prod_{v  \in V}
(1-\pi_v)^{\phi_v(\umu)- \mu_v}
\end{equation}
is proven in \cite{BFPS11}, and the corresponding relative conclusion
\begin{equation}\label{relative_version_paper}
 \frac{\Prob\left(\bigcap_{v \in V} A_v^c\right)}{\Prob\left(\bigcap_{u \in U}
A_u^c\right)} \geq \prod_{v \in V\setminus U} (1-\pi_v)^{\phi_v(\umu)-\mu_v}
\end{equation}
follows directly from the argument there.  To deduce (\ref{typical_conclusion})
and (\ref{relative_version}), observe that \[\phi_v(\umu) -\mu_v \leq \left(1 -
\pi_v\right) \phi_v(\umu),\] so that
\[
 (1-\pi_v)^{\phi_v(\umu)-\mu_v} \geq \exp\left(\phi_v(\umu) (1-\pi_v)
\log(1-\pi_v)\right)
\geq \exp( - \phi_v(\umu) \pi_v) \geq \exp(-\mu_v).
\]

\end{proof}

Recall that Theorem \ref{LLL_type_theorem} applies in the context of a
congruence system to moduli in a set $\cN$, whose prime factors lie in a set
$\cP$.  Each modulus $n \in \cN$ has a set of residues $\ba_n$, considered to
be a probabilistic event with probability $\frac{|\ba_n|}{n}$.  We require a
system of non-negative weights $\{x_p\}_{p \in \cP}$ satisfying
\[
 x_p \geq \sum_{n \in \cN: p|n} \frac{|\ba_n \bmod n| \prod_{p'|n}(1 +
x_{p'})}{n}
\]
and the conclusion is that the uncovered set $R$ has density at least
\[
 \Prob(R) \geq \exp\left(-\sum_{n \in \cN} \frac{|\ba_n \bmod n| \prod_{p|n}(1
+ x_p)}{n}\right)
\]
and that, for any $n \in \cN$, for any $b \bmod n$,
\[
 \frac{\Prob(R \cap (b\bmod n))}{\Prob(R)} \leq \frac{\exp\left(\sum_{p|n}
x_p\right)}{n}.
\]

\begin{proof}[Deduction of Theorem \ref{LLL_type_theorem}]
To deduce Theorem \ref{LLL_type_theorem} from Theorem
\ref{sharp_lovasz_theorem} we take $V$ to be the set of non-trivial square-free
products of primes in $\cP$,
\[
 V = \{v>1, \text{square-free}, p|v \Rightarrow p \in \cP\}.
\]
The event associated to $v \in V$ is the union of congruences $(\ba_n \bmod n)$
for which $\sqf(n) = v$, and this event has probability
\[
 \pi_v = \sum_{n : \sqf(n) = v} \frac{|\ba_n|}{n}.
\]
The dependency graph connects $v_1$ and $v_2$ if and only if $\GCD(v_1,
v_2)>1$.

We take the weight $\mu_v$ to be multiplicative, $\mu_v = \pi_v \prod_{p|v} (1
+ x_p)$.
This has the effect of reducing
(\ref{lovasz_condition}) at $v$ to the constraint
\begin{equation}\label{reduced_constraint}
 \prod_{p|v} (1 + x_p) \geq \phi_v(\umu).
\end{equation}
Notice that
\[
 \phi_v(\umu) = \sum_{\substack{R \subset \{v\} \cup N_v\\ \text{independent}}}
\prod_{v' \in R} \mu_{v'} \leq \prod_{p |v}\left(1 + \sum_{v' : p|v'}
\mu_{v'}\right)
\]
since each term in the sum on the left appears in the expansion of the product
on the right.  Thus if we make the condition that for each $p|Q'$,
\[
 x_p \geq \sum_{v': p|v'} \mu_{v'},
\]
which is the condition (\ref{reduced_constraint}) in the case
$v=p$,
then (\ref{reduced_constraint}) holds automatically for all $v$. In this way we
have reduced to guaranteeing the system of prime constraints
\begin{equation}\label{prime_constraints}
 \forall p|Q', \qquad x_p \geq \sum_{v': p|v'} \pi_{v'} \prod_{p':
p'|v'}(1 + x_{p'}),
\end{equation}
which is the constraint of Theorem \ref{LLL_type_theorem}.

The first conclusion, (\ref{typical_conclusion}) of Theorem
\ref{sharp_lovasz_theorem} now gives that
\begin{align*}
 \Prob\left(\bigcap_{n \in \cN} (\ba_n \bmod n)^c\right) &\geq
\exp\left(-\sum_{v \in V} \pi_v \prod_{p|v}(1 + x_p)\right)\\& =
\exp\left(-\sum_{n \in \cN} \frac{|\ba_n \bmod n| \prod_{p |n}(1 +
x_p)}{n}\right),
\end{align*}
which is the first conclusion of Theorem \ref{LLL_type_theorem}.  To get the
second conclusion, use
\begin{align*}
 \frac{\Prob(R \cap (b \bmod n))}{\Prob(R)} &\leq \frac{\Prob\left((b
\bmod n) \cap \bigcap_{n' \in \cN, (n,n')=1} (\ba_{n'} \bmod
n')^c\right)}{\Prob\left(  \bigcap_{n' \in \cN} (\ba_{n'} \bmod
n')^c\right)}\\ & = \frac{1}{n} \frac{\Prob\left( \bigcap_{n' \in \cN, (n,n')=1}
(\ba_{n'} \bmod
n')^c\right)}{\Prob\left(  \bigcap_{n' \in \cN} (\ba_{n'} \bmod
n')^c\right)}\\ & \leq \frac{1}{n} \exp\left( \sum_{n' \in \cN: (n',n)>1}
\frac{|\ba_{n'} \bmod n'| \prod_{p|n'}(1 + x_p)}{n'}\right).
\end{align*}
The last term is bounded by
\[
 \frac{1}{n} \exp\left(\sum_{p|n} \sum_{n': p|n'} \frac{|\ba_{n'} \bmod n'|
\prod_{p'|n'} (1 + x_{p'})}{n'}\right) \leq \frac{1}{n} \exp\left(\sum_{p|n}
x_p\right).
\]
\end{proof}

We now give a sufficient criterion to guarantee a good solution to the fixed
point equation governing existence of weights in Theorem \ref{LLL_type_theorem}.
Recall that we define
\[
 G_p(\ux) = \sum_{n\in \cN:p|n} \frac{|\ba_n \bmod n| \prod_{p'|n}(1 +
x_{p'})}{n}.
\]
A
trivial lower bound for a fixed point $\uG(\ux^{\fix}) = \ux^{\fix}$  is
\[
 \ux^0, \qquad x_p^0 = \frac{G_p(\uO)}{1-G_p(\uO)},
\]
and we wish to say that a fixed point lies near $\ux^0$.
The $n$th derivative $D^n \uG(\uO)$ is a multilinear map $\bigotimes^n
\ell^2(\cP) \to \ell^2(\cP)$.  Give it the usual operator norm,
\[
 \|D^n \uG(\uO)\|_{\op} = \sup_{\|v_1\|_{\ell^2}=...= \|v_n\|_{\ell^2} = 1}
\|D^n\uG(\uO)(v_1, ..., v_n)\|_{\ell^2}.
\]
The following theorem guarantees  that there
exists such a fixed point $\ux^{\fix}$ close to $\ux^0$ when there is good
control of
the operator norms of the derivatives of $D^n(\uG)(\uO)$ of $\uG$ at $\uO$. The theorem was motivated by the series of approximations made in Newton's method.

\begin{theorem}\label{finding_lovasz_theorem}
 With the notation as above, let $M>0$ be a parameter.  Assume that
 \[
  B_\infty = \|\uG(\uO)\|_{\ell^\infty} <1,
 \]
 and set $B_{2,0} = \|\ux^0\|_{\ell^2}$ and
 \[
  B_{\op}(M) = \|D\uG(\uO) - \diag(D\uG(\uO))\|_{\op} + \sum_{n=2}^\infty
\frac{M^{n-1}}{(n-1)!} \|D^n\uG(\uO)\|_{\op}< \infty.
 \]
Suppose that $\theta = \frac{B_{\op}}{1-B_{\infty}}<1$ and that
$\frac{B_{2,0}}{1-\theta} \leq M$.  Then there exists  $\ux^{\fix} = \ux^0 +
\ueps$,
$\ueps
\geq \uO$ solving the fixed point equation $\uG(\ux^{\fix}) = \ux^{\fix}$, such
that
\[
 \|\ueps\|_{\ell^2} \leq \frac{B_{2,0} \theta}{1-\theta}.
\]

\end{theorem}

\begin{proof}
 Let $F(\ux) = \uG(\ux) - \ux$ so that we seek to solve $F(\ux^{\fix}) = \uO$.
This
we
can attempt via `Newton's method'.

Let \[\cD = \diag(D\uG(\uO)) =
\diag(G_p(\uO)),
\qquad \cOD = D\uG(\uO) - \diag(D\uG(\uO)). \]
Starting from the initial guess $\ux^0$ as above, set $\ux^{i+1} = \ux^{i} +
(I-\cD)^{-1} F(\ux^{i})$.  We may also set $\ux^{-1} = \uO$, which is
consistent with this definition. A moment's thought shows that the sequence
$\ux^i$ is increasing, so that if it is bounded it converges to the desired
fixed point, and $\ueps = \sum_{i = 0}^\infty (\ux^{i+1} - \ux^i)$. Plainly
$\|(I-\cD)^{-1}\|_{\op} = \frac{1}{1-B_\infty}$, so that
\[
 \|\ux^{i+1} - \ux^i\|_{\ell^2} \leq
\frac{1}{1-B_{\infty}}\|F(\ux^i)\|_{\ell^2}.
\]
Note that $F$ is a polynomial.  Thus
\[
 F(\ux^i) = \sum_{k=0}^\infty \frac{D^k F(\uO)(\ux^i, ..., \ux^i)}{k!}.
\]
In this sum, write $DF(\uO) = (\cD - I) + \cOD$, and recall that $\ux^i
=\ux^{i-1}
+  (I-\cD)^{-1} F(\ux^{i-1})$, so that
\[
 DF(\uO) \ux^i = \cOD \ux^i + (\cD-I) \ux^{i-1} - F(\ux^{i-1}).
\]
On Taylor expanding $F(\ux^{i-1})$
we find
\begin{equation}\label{Taylor_F}
 F(\ux^i) = \cOD (\ux^i - \ux^{i-1}) + \sum_{k=2}^\infty \frac{1}{k!} (D^k
F(\uO)(\ux^i, ..., \ux^i) - D^k F(\uO)(\ux^{i-1}, ..., \ux^{i-1})).
\end{equation}
Now we impose the constraint $\|\ux^{j}\|_{\ell^2} \leq M$, which holds for
$j = 1$, and which we will verify for all $j$ by induction.  With this
assumption, by the usual trick with the triangle inequality in which we change one coordinate at a time,
\[
 \|D^k
F(\uO)(\ux^i, ..., \ux^i) - D^k F(\uO)(\ux^{i-1}, ..., \ux^{i-1})\|_{\ell^2}
\leq
k M^{k-1} \|D^k F(\uO)\|_{\op} \|\ux^{i} - \ux^{i-1}\|_{\ell^2},
\]
so that $|F(\ux^i)\|_{\ell^2} \leq B_{\op} \|\ux^i - \ux^{i-1}\|$ and
\[
 \|\ux^{i+1} - \ux^{i}\|_{\ell^2} \leq \frac{B_{\op}}{1-B_\infty}
\|\ux^i-\ux^{i-1}\|_{\ell^2} = \theta \|\ux^i-\ux^{i-1}\|_{\ell^2}.
\]
Since
\[
 \|\ux^0 - \ux^{-1}\|_{\ell^2} = \|\ux^0\|_{\ell^2} = B_{2,0},
\]
we have $\|\ux^i\| \leq \frac{B_{2,0}}{1-\theta} \leq M$ for all $i$, which
verifies the condition above.  It follows that $\|\ueps\|_{\ell^2} \leq
\frac{B_{2,0} \theta}{1-\theta}$.
\end{proof}
\subsection{Random Lov\'{a}sz weights}
For each $i = 1, 2, ...$  on (a subset of good) fibers above $R_{i-1}^* \cap
R_i$ we apply Theorem \ref{LLL_type_theorem} with moduli $\cN = \cN_{i+1}$ and
residues $\ba_n = \ba_{n,r}$.  Thus we think of the quantities from Theorems
\ref{LLL_type_theorem} and \ref{finding_lovasz_theorem} as depending upon the
random variable $r$, e.g. $\uG(\uO) = \uG(\uO,r)$, $B_\infty = B_{\infty, r}$,
$\ux^{0} = \ux^{0}(r)$.  We wish to understand properties of
the distribution of $\ux^{\fix}(r)$, but will instead define good fibers in
terms of $\ell^p$ control of $\uG(\uO,r)$, and control of $B_{\op}(M,r)$, the
other quantities of interest being controlled in terms of these. We now work to
control $B_{\op}$.

We
directly verify that the partial derivatives of $D^kG$ are given by
\[
 D_{p_1}...D_{p_k} G_p = \left\{\begin{array}{lll}
\displaystyle\sum\limits_{\substack{n \in \cN\\ p, p_1...p_k |k}} \frac{|\ba_n
\bmod n|}{n}
\displaystyle\prod\limits_{\substack{p'|n\\ p' \not \in \{p_1,...,p_k\}}} (1 +
x_{p'})
&&\text{if } p_1, ..., p_k \text{ distinct}\\\\ 0 &&
\text{otherwise}\end{array}\right..
\]

A simple bound for the operator norm of $D^k\uG(\uO)$ is
\[
 \|D^k\uG(\uO)\|_{\op} \leq \|D^k \uG(\uO)\|_{\ell^2} =
\left(\sum_{\substack{p_1, ..., p_k\\ \text{distinct}}}
\|D_{p_1}...D_{p_k}\uG(\uO)\|_{\ell^2}^2\right)^{\frac{1}{2}},
\]
which, in view of the evaluation of $D^kG$, is given by, for $k \geq 2$,
\begin{align*}
&\|D^k\uG(\uO)\|_{\op}^2 \leq  S_k \\ & S_k := \sum_{p_1, ..., p_k \text{
distinct}}\left( k \left(\sum_{\substack{n \in \cN\\
p_1...p_k|n}} \frac{|\ba_n \bmod n|}{n}  \right)^2  + \sum_{p \not
\in \{p_1, ..., p_k\}} \left(\sum_{\substack{n \in \cN\\ pp_1...p_k
|n}}\frac{|\ba_n \bmod n|}{n} \right)^2\right).
\end{align*}

In $\cOD = D\uG(\uO) - \diag(D\uG(\uO))$ the diagonal terms $p = p_1$
are missing, so that we recover the bound
\begin{align*}
 \|\cOD\|_{\op}^2 &\leq \sum_{p \neq p_1} \left(\sum_{\substack{n \in \cN\\
pp_1|n}} \frac{|\ba_n \bmod
n|}{n}\right)^2 =: S_1.
\end{align*}

By Cauchy-Schwarz, for positive weights $W_1, W_2,
W_3,...$
\begin{align*}
 B_{\op}(M)^2 \leq \left( \sum_{k=1}^\infty
\frac{W_k M^{k-1}}{ (k-1)!}
\right)\left(\sum_{k=1}^\infty \frac{ M^{k-1}}{(k-1)!} \frac{S_k}{W_k}
\right).
\end{align*}
Let $\min(\cP) \geq P+1$. We choose
\[
 W_1^2 = \frac{1}{(P\log P)^2}, \qquad \forall
k \geq 2, \quad W_k^2 =  \frac{k}{(P \log
P)^k},
\]
from which it follows

\begin{align}\notag
 B_{\op}(M)^2  &\leq \left(\frac{1}{P \log P} + \sum_{k=2}^\infty
\frac{k^{\frac{1}{2}}
M^{k-1}}{(k-1)! ( P \log P)^{\frac{k}{2}}}
\right) \\\notag &  \times
\left((P \log P) S_1 +
\sum_{k=2}^\infty \frac{ M^{k-1} ( P \log
P)^{\frac{k}{2}}}{k^{\frac{1}{2}}(k-1)!} S_k\right)
\\& =: \cC \times \cS. \label{B_op_Holder}
\end{align}

We  record bounds for $\|\uG(\uO)\|_{\ell^2}$ and $S_1, S_2, ...$
etc averaged over $r \in R_{i-1}^* \cap R_i$.
\begin{lemma}\label{expectation_lemma}
 Let $i \geq 1$.  For $r \in R_{i-1}^* \cap R_i$ consider $\cN = \cN_{i+1}$ and
$\ba_n = \ba_{n,r}$ as in the discussion above. We have the
following bounds.
\begin{align*}
\E_{r \in R_{i-1}^* \cap R_i}
\|\uG(\uO,r)\|_{\ell^2}^2  &\leq \beta_2^2(i) \prod_{P_i \leq p < P_{i+1}}
\left(1 + \frac{1}{p-1}\right)^2\sum_{P_i \leq p < P_{i+1}} \frac{1}{(p-1)^2} \\
\E_{r \in R_{i-1}^* \cap R_i} \|\uG(\uO,r)\|_{\ell^3}^3 & \leq \beta_3^3(i)
\prod_{P_i \leq p < P_{i+1}}
\left(1 + \frac{1}{p-1}\right)^3 \sum_{P_i \leq p < P_{i+1}} \frac{1}{(p-1)^3}\\
 \E_{r \in R_{i-1}^* \cap R_i}
S_1(r) & \leq \beta_2^2(i) \prod_{P_i \leq p< P_{i+1}}\left(1 + \frac{1}{p
-1}\right)^2 \left(\sum_{P_i \leq p < P_{i+1}} \frac{1}{(p-1)^2
}\right)^2
\end{align*}
and, for $k \geq 2$,
\begin{align*}
 \E_{r \in R_{i-1}^* \cap R_i} S_k(r)  &\leq \beta_2^2(i)\prod_{P_i \leq p <
P_{i+1}}\left(1 + \frac{1}{p -1}\right)^2
\\ & \qquad \times \left[k\left(\sum_{P_i \leq p <
P_{i+1}} \frac{1}{(p-1)^2}\right)^k\left(1 + \frac{1}{k} \sum_{P_i \leq p <
P_{i+1}} \frac{1}{(p-1)^2}\right)\right]
\end{align*}

\end{lemma}
\begin{proof}
 These follow directly from the convexity lemma, Lemma \ref{convexity_lemma},
and the bound, for distinct $P_i \leq p_1, ..., p_k < P_{i+1}$,
\[
 \sum_{\substack{n \in \cN_{i+1}\\ p_1 ...p_k|n}} \frac{1}{n} <
\frac{1}{(p_1-1)...(p_k-1)} \sum_{n \in \cN_{i+1}} \frac{1}{n} <
\frac{1}{(p_1-1)...(p_k-1)} \prod_{P_i \leq p < P_{i+1}} \left(1 +
\frac{1}{p-1}\right).
\]

\end{proof}

Inserting (\ref{B_op_Holder}) in the last lemma, we conclude the following
bound.
\begin{lemma}\label{B_op_averaged_lemma}
 Let $B_{\op}$ be the constant from Theorem \ref{finding_lovasz_theorem}.
Averaged over $R_{i-1}^* \cap R_i$, we have the bound
\begin{align*}
\E_{r \in R_{i-1}^* \cap R_i} B_{\op}(M)^2& \leq \cC_i  \beta_2^2(i)
\prod_{P_i \leq p < P_{i+1}} \left(1 + \frac{1}{p-1}\right)^2
\\& \times\Biggl(P_i \log P_i\Biggl(\sum_{P_i \leq p < P_{i+1}}
\frac{1}{(p-1)^2}
\Biggr)^{2}  \\ & \qquad+ \Biggl(1 + \frac{1}{P_i}\Biggr) \sum_{n=2}^\infty
\frac{n^{\frac{1}{2}} M^{n-1} (P_i\log P_i)^{\frac{n}{2}}}{(n-1)!}
\Biggl(\sum_{P_i \leq p < P_{i+1}}
\frac{1}{(p-1)^2}\Biggr)^n\Biggr)
\end{align*}
with $\cC_i$ given as above by
\[
 \cC_i = \left(\frac{1}{P_i \log P_i} + \sum_{n=2}^\infty \frac{n^{\frac{1}{2}}
M^{n-1}}{(n-1)!(P_i \log P_i)^{\frac{n}{2}}}
\right).
\]

\end{lemma}

We conclude this section with a brief discussion of how we apply Theorem
\ref{LLL_type_theorem}. Beyond demonstrating that fibers above a good
set $R_i^*$ are non-empty, the information that we wish to obtain from Theorem
\ref{LLL_type_theorem} is a bound for the bias statistics $\beta_k(i+1)$ in the
next stage of iteration. Lemma \ref{bias_stat_growth_lemma} reduces this problem
to bounding the individual biases $b_n$ of $R_i^* \cap R_{i+1}$ at $n \in
\cN_{i+1}$, and Theorem \ref{LLL_type_theorem} demonstrates that this bias is
bounded by
\begin{equation}\label{b_n_eqn}
 b_n \leq \max_{r \in R_i^*} \exp\left(\sum_{p|n} x^{\fix}_p(r)\right).
\end{equation} We
 bound this quantity in terms of the number of prime factors $\omega =
\omega(n)$ of $n$.  Thinking of $\ueps = \ueps(r)$ as a small error, we have
\[
\sum_{p|n} x^{\fix}_p(r) \leq  \|\ux^{\fix}\|_{\infty, \omega} \leq
\|\ux^0\|_{\infty,
\omega} + \sqrt{\omega}\|\ueps\|_{\ell^2},
\]
where $\|\cdot\|_{\infty, k}$ denotes the norm
\[
 \|\ux\|_{\infty, \omega} = \max_{i_1 < i_2 < ... < i_\omega} (|x_{i_1}| + ... +
|x_{i_\omega}|).
\]
Since we typically  have information regarding $\|\uG(\uO)\|_{\ell^p}$ for $p =
2$ or 3 (or both) we are led to a maximization problem of the
type,
\begin{align}\label{optimization}
 \text{given:} \qquad &  0 < B_p < 1, \; 1 \leq \omega\\ \notag
 \text{maximize:} \qquad &\|\ux^0\|_{\infty, \omega}\\ \notag
 \text{subject to:} \qquad &
\|\uG(\uO)\|_{\ell_p} \leq B_p.
\end{align}
In the case $p =2$ this may be easily solved along the following lines. It is no
loss to assume that
$\uG(\uO) \in \bR_{\geq 0}^\omega$. An
application of Lagrange multipliers gives that the
coordinates of the optimum take at most 3 values,
$0 = c_1 < c_2 \leq \frac{1}{3} \leq c_3 \leq B_2$, subject to $c_2(1-c_2)^2 =
c_3(1-c_3)^2$. When there are two
non-zero values, $c_2$ is
constrained by $c_2(1-c_2)^2 \geq B_2(1-B_2)^2$, which is only possible for a
bounded number of non-zero entries.
For large $\omega$, the optimum is $\frac{\sqrt{\omega} B_2}{1-
\frac{B_2}{\sqrt{\omega}}}$,
so that the best choice for all $\omega$ is a finite check.

The case $p = 3$ is actually simpler, because, in that case also there are at
most 2 non-zero values, and they necessarily satisfy $c_2 = 1-c_3$.

\section{Explicit calculation in initial stages}\label{initial_stage_section}
In the initial stage, we appeal to the Shearer-type Theorem, Theorem
\ref{shearer_applied_theorem}, with the primes in the range $P_0 = 4 < p
<P_1 = 222$ and we verify numerically that the condition of the theorem
holds. We also calculate the bound for bias statistics
\[
 \beta_2(1) \leq 12.25, \qquad \beta_3(1) \leq 25.
\]
The method of performing these explicit computations is described in Appendix
\ref{explicit_comp_section}.  Empirically, the barrier to ruling out an odd covering using the current method is that the optimal application in the initial stage can only accommodate a few primes, so that the resulting bounds for moments do not permit the process to continue.

Let $P_2 = 4000$.

In order to choose the good set $R_1^* \subset R_1 (= R_0^*
\cap R_1)$ we appeal to Lemmas \ref{expectation_lemma} and
\ref{B_op_averaged_lemma} to calculate, for any $C_2, C_{\op} > 0$, and
for $M = \constM$,
\begin{align*}
 &\E_{r \in R_1}\left(C_2 \|\uG(\uO,r)\|_{\ell^2}^2 + C_{\op}
B_{\op}(\constM,r)^2\right)
\\
 &\leq C_2 \beta_2^2(1) \prod_{222 \leq p < 4000} \left(1 +
\frac{1}{p-1}\right)^2
\sum_{222 \leq p < 4000} \frac{1}{(p-1)^2}\\
&  + C_{\op} \cC_1 \beta_2^2(1) \prod_{222 \leq p < 4000} \left(1 +
\frac{1}{p-1}\right)^2\\ &\qquad\qquad\times
\Biggl[ (222 \log 222) \sum_{222 \leq p < 4000} \frac{1}{(p-1)^2}
\\ & \qquad\qquad\qquad +\frac{223}{222}
\sum_{n=2}^\infty
\frac{n^{\frac{1}{2}} \constM^{n-1} (2 \cdot 222^2 \log
222)^{\frac{n}{2}}}{(n-1)!} \Biggl(\sum_{222 \leq p < 4000}
\frac{1}{(p-1)^2}\Biggr)^n \Biggr]
\end{align*}
and
\[
 \cC_1 = \frac{1}{222(\log
222)} + \sum_{n=2}^\infty \frac{n^{\frac{1}{2}}
\constM^{n-1}}{(n-1)! (222  \log 222)^{\frac{n}{2}}} .
\]
We calculate numerically that
\[
 \E_{r \in R_1}\left(C_2\|\uG(\uO,r)\|_{\ell^2}^2 + C_{\op}
B_{\op}(M,r)^2\right) <
C_{2}\cdot\constEBTwo + C_{\op} \cdot \constEBop.
\]
We choose $C_2 = \frac{\constRhoTwo}{\constEBTwo}$ and $C_{\op} =
\frac{\constRhoOp}{\constEBop}$, so that the above inequality reads
\[
 \E_{r \in R_1}\left(C_2\|\uG(\uO,r)\|_{\ell^2}^2 + C_{\op}
B_{\op}(\constM,r)^2\right) < 1.
\]

We say that $r \in R_1$ is good if
\[
 \|C_2 \uG(\uO,r)\|_{\ell^2}^2 +  C_{\op} B_{\op}(\constM,r)^2 \leq
\frac{1}{1-\constPiGood}.
\]
By Markov's inequality, $\pi_{\good}(1) \geq \constPiGood$.

Evidently, for all $r$,
\[
 \|\uG(\uO, r)\|_{\ell^2}^2 \leq \frac{1}{C_2 (1-\constPiGood)} <
\constMBTwoTwo.
\]
We save a little extra ground by conditioning on the actual size of
$\|\uG(\uO, r)\|_{\ell^2}$.  Let $K = \constSplice$ be a parameter.  For $1 \leq
j \leq K$ we say that $r \in R_1^*$ is in bin $\cB_j$ if
\[
\|\uG(\uO, r)\|_{\ell^2}^2 \in  \left(\frac{j-1}{K}, \frac{j}{K}\right] \cdot
\frac{1}{C_2(1-0.3)}.
\]
For $r \in \cB_j$ we have 
\[
B_{\op}(\constM, r)^2 \leq \frac{K-j+1}{K}
\cdot \frac{1}{C_{\op}(1-\constPiGood)}.
\]
Abusing notation, for $r \in \cB_j$ we write quantities
depending upon $r$ as depending upon $j$ instead so, we use $\uG(\uO, j)$, $B_2(j)$,
$B_{\op}(j)$, and so forth.

In each bin we update
\[
 B_\infty(j) = \|\uG(\uO,j)\|_{\ell^\infty}\leq B_2(j)
\]
and thus
\[
 B_{2,0}(j) = \|\ux^0\|_{\ell^2} \leq \frac{B_2(j)}{1-B_2(j)},
\]
and $\theta(j) \leq \frac{B_{\op}(j)}{1-B_{2}(j)}.$  We check numerically,
bin-by-bin, that for all bins,
\[
 \frac{B_{2,0}(j)}{1-\theta(j)}  < \constM  = M
\]
so that the condition of Theorem \ref{finding_lovasz_theorem} is met.  In
particular, each good fiber is non-empty.

Again, we apply Theorem
\ref{finding_lovasz_theorem} bin-by-bin so that, in each bin we obtain
a bound of
\[
\|\ueps(j)\|_{\ell^2} \leq \frac{B_{2,0}(j) \theta(j)}{1-\theta(j)}.
\]

Beginning from the information $ B_2^2(j) \leq \frac{j}{K}\constMBTwo$,
we solve the optimization problem (\ref{optimization}) for $\omega = 1, 2, 3,
...$. As we have already commented, for each $\omega$ the optimal solution has
no more than two non-zero values among the $x_p$.  When it has the two values
$c_1, c_2$ these satisfy for some positive integers $a\leq b$, $a+b \leq
\omega$,
\[a c_1^2 + bc_2^2 \leq \frac{j}{K} \constMBTwoTwo, \qquad c_1(1-c_1)^2 =
c_2(1-c_2)^2.\]  It transpires that this possibility occurs only for $\omega
\leq 4$ and when $a = 1$.  For $\omega \geq \constOmegaThreshold$, the
optimum in each bin is given by
\[
 \|\ux^0(j)\|_{\infty, \omega} \leq \omega \frac{B_2(j)}{\sqrt{\omega} -
B_2(j)}, \qquad \|\ux^{\fix}(j)\| \leq \omega \frac{B_2(j)}{\sqrt{\omega} -
B_2(j)} + \sqrt{\omega}\|\ueps(j)\|_{\ell^2}.
\]
Obviously $B_2(j) \leq B_2(K) \leq \constMBTwoTwo^{\frac{1}{2}} = \constMBTwo$,
and we find
\[
 \sup_j \|\ueps(j)\|_{\ell_2} \leq \constEpsMax.
\]
Resulting bounds for
$
 \sup_j \|\ux^{\fix}\|_{\infty,
\omega}
$
are recorded  in the following table

 \begin{center}
\begin{tabular}{|c||c|}
\hline
$\omega$  & $\|\ux^{\fix}\|_{\infty,\omega}$
\\
\hline
1 & \constFOa \\
2 & \constFOb\\
3 & \constFOc\\
4 & \constFOd\\
5 & \constFOe\\
6  & \constFOf\\
7 & \constFOg\\
8  & \constFOh\\
9  & \constFOi\\
10  & \constFOj\\
$\omega > 10$ & $\omega
\frac{\constMBTwo}{\sqrt{\omega}-\constMBTwo} +
\constEpsMax \sqrt{\omega}$ \\
\hline
 \end{tabular}
 \end{center}

We can thus update the bound for bias statistics $\beta_k(2)$ according to
Lemma \ref{bias_stat_growth_lemma}.  Write, for $k = 1, 2, 3, ...$,
\[
 \tau_k(p) = \sum_{i=1}^\infty \frac{(i+1)^k - i^k}{p^i}.
\]
for the local factor at $p$ that occurs at the $k$th bias statistic.  Then the
new bound becomes
 \[
 \beta_k^k(2) \leq \frac{\beta_k^k(1)}{\pi_{\good}(1)} \left(1 +
\sum_{j=1}^\infty \exp\left(\|\ux^{\fix}\|_{\infty, j}\right)
e_j\left(\tau_k(p): P_1 \leq p < P_2\right)\right)
 \]
where $e_j$ indicates the $j$th elementary symmetric function.  For large $j$
we use the bound $e_j(\utau) \leq \frac{e_1(\utau)^j}{j!}$.
In this way we calculate that
\[
 \beta_2(2) < \constbetaTwonew, \qquad
\beta_3(2)  < \constbetaThreenew.
\]

\section{Asymptotic estimates}
Recall that $P_2 = 4000$.
For all $i \geq 2$ we let $P_{i+1}=P_i^{1.5}$.  In
this section we use the following explicit estimates for sums and products over primes, which hold for $i \geq 2$.
\begin{align*}
 \prod_{P_i \leq p < P_{i+1}} \left(\frac{p}{p-1} \right) &< (\constProdFactor)(1.5) = \constProdBound\\
 \sum_{P_i \leq p < P_{i+1}} \frac{1}{(p-1)^2} & < \frac{\constSumTwoFactor}{P_i \log P_i}\\
 \sum_{P_i \leq p < P_{i+1}} \frac{1}{(p-1)^3} & < \frac{\constSumThreeFactor}{2P_i^2 \log P_i}.
\end{align*}
These are verified in Appendix \ref{prime_sum_appendix}.

For the remainder of the argument our
inductive assumption is, for $i
\geq 2$,
\begin{align}\label{inductive_assumption}
 \beta_2(i) &\leq \constbetaTwoNewratio\cdot(P_i \log P_i)^{\frac{1}{2}}\\
 \beta_3(i) &\leq \constbetaThreeNewratio\cdot(2P_i^2 \log
P_i)^{\frac{1}{3}}.\notag
\end{align}
Note that both of these hold at $i = 2$.

Setting $M = \constMIter$ in Theorem \ref{finding_lovasz_theorem}, we estimate
\[
 \E_{r \in R_{i-1}^* \cap R_i} (C_3\|\uG(\uO,r)\|_{\ell^3}^3 + C_2
\|\uG(\uO,r)\|_{\ell^2}^2 + C_{\op} B_{\op}(\constMIter,r)^2).
\]
Appealing to Lemma \ref{B_op_averaged_lemma} we bound the sums in  $\cC$ and
$\E\cS$, implicitly defined in (\ref{B_op_Holder}), by
\[
 \cC_i \leq \cC_2 =  \left(\frac{1}{4000 \log 4000} + \sum_{n=2}^\infty
\frac{n^{\frac{1}{2}}
M^{n-1}}{(n-1)! (4000 \log 4000)^{\frac{n}{2}}}
\right) \leq \constcCTwo,
\]
\begin{align*}
 &\E\cS_i
 \leq  (\constProdBound)^2 \beta_2^2(i) \left(\frac{(\constSumTwoFactor)^2}{P_i \log P_i} +
\left(1+\frac{1}{P_i}\right)\sum_{n=2}^\infty
\frac{n^{\frac{1}{2}}
M^{n-1}(\constSumTwoFactor)^n}{(n-1)! (P_i\log P_i)^{\frac{n}{2}}}
\right)
\\& \leq (\constProdBound)^2(\constbetaTwoNewratio)^2  \left((\constSumTwoFactor)^2 + \frac{4001}{4000}\sum_{n=2}^\infty
\frac{n^{\frac{1}{2}}
M^{n-1}(\constSumTwoFactor)^n}{(n-1)! (4000 \log 4000)^{\frac{n}{2}-1}}
\right)
\\& \leq \constcSTwo.
\end{align*}

 Combined with the
asymptotics of $\E \|\uG(\uO)\|_{\ell^p}^p$ from Lemma \ref{expectation_lemma},
\begin{align*}
  \E_{r \in R_{i-1}^* \cap R_i} \|\uG(\uO)\|_{\ell^3}^3 &\leq \prod_{P_i \leq p < P_{i+1}} \left(\frac{p}{p-1}\right)^3 \beta_3^3(i) \sum_{P_i \leq p < P_{i+1}} \frac{1}{(p-1)^3}
  \\&< (\constProdBound)^3 (\constSumThreeFactor)(\constbetaThreeNewratio)^3 < \constEBThreeThreeIter
\\ \E_{r \in R_{i-1}^* \cap R_i} \|\uG(\uO)\|_{\ell^2}^2 &\leq \prod_{P_i \leq p < P_{i+1}}\left( \frac{p}{p-1} \right)^2 \beta_2^2(i) \sum_{P_i \leq p < P_{i+1}} \frac{1}{(p-1)^2} \\& <
(\constProdBound)^2 (\constSumTwoFactor)
(\constbetaTwoNewratio)^2 < \constEBTwoTwoIter
\end{align*}
we deduce
\begin{align*}
 \E_{r \in R_{i-1}^* \cap R_i} &\left(C_3\|\uG(\uO)\|_{\ell^3}^3 + C_2
\|\uG(\uO)\|_{\ell^2}^2 + C_{\op} B_{\op}(\constMIter,r)^2\right) \\&\leq
\constEBThreeThreeIter C_3 + \constEBTwoTwoIter C_2 + \constEBopTwoIter C_{\op}.
\end{align*}
Choose $C_3 = \frac{\constrhoThreeIter}{\constEBThreeThreeIter}$, $C_2 =
\frac{\constrhoTwoIter}{\constEBTwoTwoIter}$, $C_{\op} =
\frac{\constRhoOpIter}{\constEBopTwoIter}$ so that the expectation is bounded by
1.
As before, declare $r \in R_{i-1}^* \cap R_i$ to be good if
\[
 C_3\|\uG(\uO,r)\|_{\ell^3}^3 +C_2
\|\uG(\uO,r)\|_{\ell^2}^2 + C_{\op} B_{\op}(\constMIter,r)^2 \leq
\frac{1}{1-\constPiGood}.
\]
Evidently $\pi_{\good}(i) \geq \constPiGood$, and for good $r$,
\begin{align*}\|\uG(\uO,r)\|_{\ell^\infty}^3&\leq \|\uG(\uO,r)\|_{\ell^3}^3
\leq \frac{1}{C_3(1-\constPiGood)} < \constMBThreeThreeIter, \\
\|\uG(\uO,r)\|_{\ell^2}^2 &\leq \frac{1}{C_2(1-\constPiGood)} <
\constMBTwoTwoIter,
\end{align*}
but, again, we bin to get a stronger result.

For  $K = \constSplice$ and
integers $0 < i,j$, $i+j \leq K+1$ let the bin $\cB_{i,j}$ be those $r$ for
which
\[
 \|\uG(\uO,r)\|_{\ell^3}^3 \in \left(\frac{i-1}{K} ,\frac{i}{K}
\right]\frac{1}{C_3(1-0.3)}, \qquad \|\uG(\uO,r)\|_{\ell^2}^2 \in
\left(\frac{j-1}{K} ,\frac{j}{K} \right]\frac{1}{C_2(1-0.3)}.
\]
For $r \in \cB_{i,j}$ we have
\[
 B_{\op}^2(r)\leq \frac{K-i-j+1}{K} \frac{1}{C_{\op}(1-\constPiGood)}.
\]

We proceed much as before (again replacing $r$ with $i,j$ in each argument)
updating bin-by-bin
\[
 B_{2,0}(i,j) \leq \frac{\|\uG(\uO,i,j)\|_{\ell^2}}{1-
\|\uG(\uO,i,j)\|_{\ell^3}},
\]
and
\[
 \theta(i,j) \leq \frac{B_{\op}(\constMIter,i,j)}{1- \|\uG(\uO,
i,j)\|_{\ell^3}}.
\]
We check bin-by-bin that
\[
 \frac{B_{2,0}(i,j)}{1-\theta(i,j)}< \constMIter = M
\]
so that our choice of $M = \constMIter$ in Theorem \ref{finding_lovasz_theorem}
is
valid.

In each bin we solve the optimization problem
(\ref{optimization}) with $p = 3$, and we find that for all $\omega \geq 1$ and
for all bins the optimum is
\[
 \|\ux^0(i,j)\|_{\infty, \omega} \leq \omega^{\frac{2}{3}}
\frac{\frac{i}{K}
\constMBThreeIter}{1-\frac{\frac{i}{K}\constMBThreeIter}{\omega^{\frac{1}{3}}}}
,
\]
so that we guarantee
\[
 \|\ux^{\fix}(i,j)\|_{\infty, \omega} \leq
\frac{\frac{i}{K}\constMBThreeIter
\omega^{\frac{2}{3}}}{1-\frac{\frac{i}{K}\constMBThreeIter}{\omega^{\frac{1}{3}
} } } +
\|\ueps(i,j)\|_{\ell^2}
\omega^{\frac{1}{2}}.
\]
We calculate
\[
 \max_{i,j} \|\ueps(i,j)\|_{\ell^2} \leq \constMaxEpsIter.
\]
Thus we find the following bounds.

 \begin{center}
\begin{tabular}{|c||c|}
\hline
$\omega$  & $\|\ux^{\fix}\|_{\infty,\omega}$
\\
\hline
1 & \constFOaIter \\
2 & \constFObIter\\
3 & \constFOcIter\\
4 & \constFOdIter\\
5 & \constFOeIter\\
6  & \constFOfIter\\
7 & \constFOgIter\\
8  & \constFOhIter\\
9  & \constFOiIter\\
10  & \constFOjIter\\
$\omega > 10$ & $
\frac{\constMBThreeIter \omega}{\omega^{\frac{1}{3}}-\constMBThreeIter} +
\constMaxEpsIter \sqrt{\omega}$ \\
\hline
 \end{tabular}
 \end{center}

In Appendix \ref{prime_sum_appendix} we verify that, for $i \geq 2$,
\[
 e_1(\tau_2(p)) \leq 3 \log 1.5 + \constTauTwoOffset < \constTauTwoBound, \qquad e_1(\tau_3(p)) \leq 7 \log 1.5 + \constTauThreeOffset < \constTauThreeBound.
\]
Hence,
\[
 e_j(\tau_2(p): P_i \leq p < P_{i+1}) \leq \frac{e_1(\tau_2(p))^j}{j!} <
\frac{(\constTauTwoBound)^j}{j!}, \qquad e_j(\tau_3(p)) <
\frac{(\constTauThreeBound)^j}{j!}
\]
and we find
\[
 \frac{\beta_2^2(i+1)}{\beta_2^2(i)} \leq \frac{1}{\constPiGood} \left(1 +
\sum_{\omega=1}^\infty
\exp\left(\max_{i,j} \|\ux^{\fix}(i,j)\|_{\infty, \omega}\right) \frac{(\constTauTwoBound)^\omega}{\omega!}\right) < \constbetaTwoGrowthIter
\]
and
\[
  \frac{\beta_3^3(i+1)}{\beta_3^3(i)} \leq \frac{1}{\constPiGood} \left(1 +
\sum_{\omega=1}^\infty
\exp\left(\max_{i,j}\|\ux^{\fix}(i,j)\|_{\infty, \omega}\right) \frac{(\constTauThreeBound)^\omega}{\omega!}\right) < \constbetaThreeGrowthIter.
\]
On the other hand,
\[
 \frac{P_{i+1}\log P_{i+1}}{P_i \log P_i} \geq \constloglengthiter \cdot
P_i^{\frac{1}{2}} \geq
\constloglengthiter \cdot 4000^{\frac{1}{2}} > 94
\]
and
\[
 \frac{ P_{i+1}^2 \log P_{i+1}}{P_i^2 \log P_i} \geq \constloglengthiter \cdot
P_i \geq 6000,
\]
so that (\ref{inductive_assumption}) is preserved, which completes the proof by
induction.

\newpage
\appendix

\section{Symmetric functions}\label{explicit_comp_section}

We briefly describe how we performed the calculations in the initial stage of
the argument, see Section \ref{initial_stage_section}.  There we appealed to
Theorem \ref{shearer_applied_theorem}, which is Theorem
\ref{initial_shearer_theorem}
with set $[n]$ identified with $\{p: 4 < p < 222\} = p_1 < p_2 < ...< p_n$ and
weights  $\pi_p = \frac{1}{p-1}$.  We identify square-free number
$m$ with the set of its prime factors.  The Shearer functions
\[
 \rho(p_1) > \rho(p_1p_2) > ...>\rho(p_1...p_n)
\]
are easily computed via
\[
 \rho(p_1...p_j) = \sum_{i = 0}^j X(i) e_i(\pi_{p_1}, ..., \pi_{p_j}).
\]
with the $e_i$ elementary symmetric functions (take $e_0=1$), see
\cite{SZ91}.

The bias statistics are also not difficult to bound.  Recall that $Q = \LCM(m: m
\in \cM)$ and that $Q_1$ is the part of $Q$ composed of primes less than $P_1 =
222.$  The $k$th bias statistic is
\[
 \beta_k^k(1) = \sum_{m |Q_1} \ell_k(m) \max_{b \bmod m} \frac{|R_1 \cap(b\bmod
m) \bmod Q_1|}{|R_1 \bmod Q_1|}.
\]
Let $\sqf(m) = \prod_{p \in S} p =: m_S$.  Appealing to
(\ref{shearer_progression}) of Theorem \ref{shearer_applied_theorem} we have
\[
 \frac{|R_1 \cap(b\bmod
m) \bmod Q_1|}{|R_1 \bmod Q_1|} \leq \frac{1}{m} \frac{\rho([n] \setminus
S)}{\rho([n])},
\]
so that
\begin{align}\notag
 \beta_k^k(1)& \leq \frac{1}{\rho([n])} \sum_{S \subset[n]} \rho([n]\setminus S)
\sum_{m : \sqf(m) = m_S} \frac{\ell_k(m)}{m}\\ &\leq \frac{1}{\rho([n])} \sum_{S
\subset [n]} \rho([n] \setminus S) \prod_{s\in S} \left(\sum_{j = 1}^\infty
\frac{\ell_k(p_s^j)}{p_s^j} \right).\label{sym_function_bias}
\end{align}
Recall that we define
\[
\tau_{k, i} = \tau_k(p_i) = \sum_{j=1}^\infty \frac{\ell_k(p_i^j)}{p_i^j} =
\sum_{j=1}^\infty \frac{(j+1)^k - j^k}{p_i^k} =
\left[\left(\frac{1}{x}-1\right)\left(x \frac{\partial}{\partial x}\right)^k
\frac{1}{1-x} -1\right]_{x = \frac{1}{p_i}}.
\]
Define for $i+j \leq n$ the mixed symmetric functions $f_{i,j}(\upi, \utau_k)$
by
\[
 f_{i,j}(\upi, \utau_k) = \frac{i!j!(n-i-j)!}{n!} \sum_{\sigma \in \Sym([n])}
\upi_{\sigma(1)}...\upi_{\sigma(i)} \utau_{k, \sigma(i+1)}...\utau_{k,
\sigma(i+j)},
\]
or, equivalently, by
\[
 \sum_{0 \leq i + j \leq n} f_{i,j}(\upi, \utau_k) x^{i}y^j = \prod_{i =
1}^n(1 + x\pi_i + y \tau_{k,i}).
\]
The sum of (\ref{sym_function_bias}) is a linear combination of the mixed
symmetric functions $f_{i,j}(\upi, \utau_k)$
\[
 \sum_{S
\subset [n]} \rho([n] \setminus S) \prod_{s\in S} \tau_{k,s} = \sum_{0 \leq i+j
\leq n} X(i) f_{i,j}(\upi, \utau_k),
\]
and so is rapidly computable.

\section{Explicit prime number estimates}\label{prime_sum_appendix}
In this appendix we sketch proofs for explicit bounds on well-known prime sums and products.
Recall $P_2 = 4000$ and, for $i \geq 2$, $P_{i+1}=P_i^{1.5}$.  In particular, no $P_i$ is prime.
Let $\gamma$ denote the Euler-Mascheroni constant.
Dusart \cite{D07}, Theorem 6.12 proves the following estimate.
\begin{theorem}
 For $x > 1$ we have
 \[
  e^\gamma (\log x)\left(1 - \frac{0.2}{(\log x)^2}\right) < \prod_{ p \leq x} \frac{p}{p-1}
 \]
and, for $x \geq 2973$ we have
\[
 \prod_{p \leq x} \frac{p}{p-1} < e^\gamma (\log x)\left(1 + \frac{0.2}{(\log x)^2}\right).
\]

\end{theorem}
As a consequence, we obtain
\begin{corollary}
 For $i \geq 2$ we have
 \[
  \prod_{P_i \leq p < P_{i+1}} \left(\frac{p}{p-1}\right) < (1.5)(\constProdFactor).
 \]
\end{corollary}

For the sums of reciprocals of squares of primes, we have the following estimate.
\begin{proposition}
 For $i \geq 2$,
 \[
  \sum_{P_i \leq p < P_{i+1}} \frac{1}{(p-1)^2} < \frac{\constSumTwoFactor}{P_i \log P_i},
 \]
 and
 \[
  \sum_{P_i \leq p < P_{i+1}} \frac{1}{(p-1)^3} < \frac{\constSumThreeFactor}{2P_i^2 \log P_i}.
 \]

\end{proposition}
\begin{proof}
We prove only the first inequality, as the second is similar.  One easily checks
\[
 \sum_{P_2 \leq p < P_3} \frac{1}{(p-1)^2} < \frac{1}{P_2 \log P_2}, \qquad \sum_{P_3 \leq p < P_4} \frac{1}{(p-1)^2} < \frac{1}{P_3 \log P_3}.
\]
For $i \geq 4$ use $\frac{0.99999997}{(p-1)^2} < \frac{1}{p^2}$ so that, for $x \geq P_4$,
\[
 0.99999997 \sum_{p \geq x}\frac{1}{(p-1)^2} < \sum_{p \geq x} \frac{1}{p^2} = -\frac{\theta(x)}{x^2 \log x} + \int_x^\infty \frac{\theta(y)}{y^3}  \frac{1 + 2 \log y}{(\log y)^2} dy.
\]
By \cite{D07}, we have the inequality $|\theta(x)-x| < 0.2 \frac{x}{(\log x)^2}$ for $x \geq 3594641$.  In particular, $0.99913 x < \theta(x) < 1.00088 x$ in this range.  Also, Lemma 9 of \cite{RS62} yields
\[
 \int_x^\infty \frac{1 + 2 \log y}{y^2 (\log y)^2} dy < \frac{2}{x \log x}.
\]
Combined, these estimates give the claim.
\end{proof}

Recall that we define $\tau_k(x) = \sum_{i=1}^\infty \frac{(i+1)^k - i^k}{x^i}$. We have
\[
 \tau_2(x) = \frac{3x - 1}{(x-1)^2}, \qquad \tau_3(x) = \frac{7x^2 - 2x + 1}{(x-1)^3}.
\]
\begin{proposition}
 For $i \geq 2$,
 \[
  \sum_{P_i \leq p < P_{i+1}} \tau_2(p) < 3 \log 1.5 + \constTauTwoOffset,
 \]
 and
 \[
  \sum_{P_i \leq p < P_{i+1}} \tau_3(p) < 7 \log 1.5 + \constTauThreeOffset.
 \]
\end{proposition}
\begin{proof}
 For $i = 2, 3$ this is verified directly.  For $x \geq P_4$ this is a consequence of the following estimate of \cite{D07}, Theorem 6.11.
\end{proof}

\begin{theorem}
 There is a constant $B$, such that, for $x > 10372$,
 \[
  \left| \sum_{p \leq x} \frac{1}{p} - \log \log x - B\right| \leq \frac{1}{10 (\log x)^2} + \frac{4}{15 (\log x)^3}.
 \]

\end{theorem}

\section{Theorems of Lov\'{a}sz and Shearer-type}\label{shearer_appendix}

The Lov\'{a}sz local lemma considers the following scenario.  In a probability
space $\cX$ there are events $\{A_v\}_{v \in V}$ with dependency graph $G = (V,
E)$, that is, $A_v$ is independent of the $\sigma$-algebra $\sigma(A_{w} : (v,w)
\not \in E)$. One seeks a positive lower bound for $\Prob(\bigcap_{v \in V}
\overline{A_v})$.  The local lemma guarantees that if there exist weights  $1 >
x_v \geq \Prob(A_v)$ satisfying
\[
 \forall v\in V, \qquad x_v \prod_{w: (v,w) \in E} (1-x_w) \geq \Prob(A_v)
\]
then
\[
 \Prob\left(\bigcap_{v \in V} \overline A_v\right) \geq \prod_{v \in V} (1-x_v).
\]
Shearer \cite{S85} gives an optimal  bound of the above type via the independent
set polynomial
\[
 \Xi\left(z_v: {v \in V}\right) = 1 + \sum_{n = 1}^\infty
\frac{1}{n!}\sum_{\substack{(v_1,..., v_n) \in V^n\\ \forall i \neq j, \; v_i
\sim v_j}} z_{v_1}...z_{v_n}.
\]
where $(v_i \sim v_j)$ means $v_i \neq v_j$ and $(v_i, v_j) \not \in E$.
\begin{theorem*}[Shearer's Theorem]
 Given $S \subset V$ let $\Xi_S(z_v: v \in V)$ denote $\Xi$ with arguments $z_v:
v \in S$ replaced by 0.  Subject to
 \[
 \forall S \subset V, \qquad \Xi_S(-\Prob(A_v): v \in V) \geq 0
 \]
it holds
 \[
  \Prob\left(\bigcap_{v \in V} \overline{A_v}\right) \geq \Xi(-\Prob(A_v): v \in
V).
 \]
\end{theorem*}
\noindent
Although Shearer's Theorem is tight, evaluating the independent set polynomial
is difficult and so there remains interest in finding statements  of a similar
type to the local lemma, which is more easily applied.

One way to reduce the complexity of Shearer's theorem is to organize the events
into collection of cliques.  We consider the scenario in which graph $G = (V,
E)$ is covered by a collection of cliques  $\cK$, that is, $E = \bigcup_{K \in
\cK} E_K$.  For $v \in V$, let \begin{equation}\label{def_p_v}
p(v) = \{K: v \in K\}.   \end{equation}
  We make the assumption that $p(v)$ uniquely determines $v$ and we assume that
all vertices have self-loops.  What is the same, we take $V = \cP(\cK)\setminus
\{\emptyset\}$ to be the collection of all non-empty subsets of $\cK$, and, for
$S_1, S_2 \in V$, set $(S_1, S_2) \in E$ if and only if $S_1 \cap S_2 \neq
\emptyset$.
Consider vertex variables $(z_v)_{v \in V}$ and clique variables $(\theta_K)_{K
\in \cK}$.  For $v \in V$, set also $\theta_v = \prod_{K: v \in K} \theta_K$.
The clique partition function is defined to be
\[
 \Xi(\vv, \vtheta) = 1 + \sum_{n \geq 1} \frac{1}{n!} \sum_{\substack{v_1, ...,
v_n\\ \text{indep.\ in } G}} z_{v_1}...z_{v_n} \theta_{v_1}...\theta_{v_n}.
\]
Evidently $\Xi(\vv, \vtheta)$ specializes to $\Xi(\vv)$ at $\vtheta = 1$.  Using
this, we prove a clique version of Shearer's theorem.
\begin{theorem}[Clique Shearer Theorem]
 Let events $\{A_v: v \in V\}$ in probability space $\cX$ have dependency graph
$G= (V,E)$ covered by cliques $\cK$ as above.  For $S \subset \cK$ define event
$B_S = \bigcup_{v: p(v) \subset S} A_v$.  Subject to the condition
 \[
  \forall S \subset \cK, \qquad \Xi(-\Prob(A_v), 1_S) > 0
 \]
 we have for all $\emptyset \subset S \subset T \subset \cK$,
 \[
  \Prob(\overline{B_T}| \overline{B_S}) \geq \frac{\Xi(-\Prob(A_v),
1_T)}{\Xi(-\Prob(A_v), 1_S)}.
 \]
\end{theorem}
\begin{rem}
 As compared to Shearer's Theorem, the clique Shearer Theorem has the advantage
that the number of conditions which must be checked is  exponential in the
number of cliques, rather than in the number of vertices.
\end{rem}

\begin{proof}
 The proof is by induction.  Let $S \subset \cK$ and suppose the conclusion
holds for subsets $T \subset S$.  Let $K \in \cK \setminus S$. Then
 \begin{align*}
  \frac{\Prob(\overline{B_{S\cup\{K\}}})}{\Prob(\overline{B_S})} &\geq 1 -
\sum_{\substack{w: K \in p(w)\\ p(w) \subset S \cup \{K\}}} \frac{\Prob(A_w \cap
 \overline{B_S})}{\Prob(\overline{B_S})} \\
  &\geq 1- \sum_{\substack{w: K \in p(w)\\ p(w) \subset S \cup \{K\}}}
\frac{\Prob(A_w) \Prob(\overline{B_{S \setminus p(w)}})}{\Prob(\overline{B_S})}.
\\
 &\geq 1 - \sum_{\substack{w: K \in p(w)\\ p(w) \subset S \cup \{K\}}}\Prob(A_w)
\frac{\Xi(-\Prob(A_v), 1_{S \setminus p(w)})}{\Xi(-\Prob(A_v), 1_S)}\\&=
\frac{\Xi(-\Prob(A_v), 1_{S\cup \{K\}})}{\Xi(-\Prob(A_v), 1_S)}.
\end{align*}
\end{proof}

As a consequence we obtain a proof of a generalization of Theorem \ref{shearer_applied_theorem}.
\begin{theorem*}[Shearer-type theorem]
 Suppose we have a probability space and a parameter $\theta \geq 1$.  Let $[n]
= \{1, 2, ..., n\}$, and assume
that for each $1 \leq i \leq n$ there is a weight $\pi_i$ assigned,
satisfying $\frac{1}{\theta} \geq
\pi_1 \geq \pi_2 \geq ... \geq \pi_n \geq 0$.  Let the sets $\emptyset
\neq T \subset [n]$ index events $A_T$ each having probability
\[
 0 \leq \Prob(A_T) \leq \theta \prod_{t \in T} \pi_t:= \pi_T.
\]
Assume that $A_T$ is  independent of  $\sigma(\{A_S: S \subset [n], S\cap T
= \emptyset\})$, so that a valid dependency graph for the events $\{A_T:
\emptyset \neq T \subset[n]\}$ has an edge between $S \neq T$ whenever $S \cap
T \neq \emptyset$.

Define $\rho_\theta(\emptyset) =1$, and given $\emptyset \neq T \subset [n]$,
set
\[
 \rho_\theta(T) = 1 - \sum_{\emptyset \neq S_1 \subset T} \pi_{S_1} +
\sum_{\substack{
\emptyset \neq S_1 , S_2 \subset T\\ S_1 < S_2 \text{ disjoint}}}
\pi_{S_1}\pi_{S_2}- \sum_{\substack{ \emptyset \neq S_1, S_2, S_3 \subset
T\\S_1 < S_2 < S_3 \text{ disjoint}}} \pi_{S_1}\pi_{S_2}\pi_{S_3} + ....
\]
Suppose that $
 \rho_\theta([1]) \geq \rho_\theta([2]) \geq ... \geq \rho_\theta([n]) > 0.$
Then for any $\emptyset \neq T \subset [n]$,
\begin{equation*}
 \Prob\left(\bigcap_{\emptyset \neq S \subset T} A_S^c\right) \geq
\rho_\theta(T) > 0
\end{equation*}
and,  for any $ T_1 \subset T_2 \subset [n]$,
\begin{equation*}
 \frac{\Prob\left(\bigcap_{\emptyset \neq S \subset T_2}
A_S^c\right)}{\Prob\left(\bigcap_{\emptyset \neq S \subset T_1} A_S^c\right)}
\geq \frac{\rho_\theta(T_2)}{\rho_\theta(T_1)}.
\end{equation*}

\end{theorem*}

\begin{proof}
 It is observed in \cite{SZ91} that $\rho_\theta(T)$ may be expressed as a
linear
combination of elementary symmetric functions in $\{\pi_t: t \in T\}$.  Indeed,
if $B(m,j)$ denotes the generalized Bell number, that is, the number of ways of
partitioning a set of size $m$ into $j$ parts then, setting $|T| = M$ and
making the convention $e_0(\underline{\pi}) = 1$,
\[
 \rho_\theta(T) = 1 + \sum_{m= 1}^M \left(\sum_{j=1}^m
(-\theta)^j B(m,j)\right)e_m(\upi) :=\sum_{i = 0}^M X_\theta(i)
e_i(\underline{\pi}),
\]
where $X_\theta$ satisfies the recurrence
\[
 X_\theta(0) = 1, \qquad \forall i \geq 1, \; X_{\theta}(i) = -\theta
\sum_{j=0}^{i-1}\binom{i-1}{j} X_\theta(j).
\]

In particular, as exploited in \cite{ScSo05}, $\rho(T)$ is affine linear in each
variable $\pi_t$.

We check that under the given conditions, $\rho_\theta(T)>0$ for any $T \subset
[n]$, which reduces this theorem to the clique Shearer Theorem.

Given vectors $\underline{x}, \underline{y} \in \bR^m$, say that $\underline{x}
\leq \underline{y}$ if $x_i \leq y_i$ for each $i$.  By induction, we show that
for any $1 \leq m \leq n$ and for $\underline{0}  \leq  \underline{x} \leq
\upi$, $\rho_\theta(\underline{x}) \geq \rho_\theta(\upi) > 0$, from which the
case for $T$
follows since the $\pi_i$ are
decreasing.

When $m = 1$, $\rho_\theta(\pi_1) = 1-\theta\pi_1 \leq 1-\theta x_1 =
\rho_\theta(x_1)$. Given $m
> 1$, assume inductively the statement for all $m'<m$.

Note that, by hypothesis, we have
$\rho_\theta(\pi_1, ..., \pi_{m-1}) \geq \rho_\theta(\pi_1, ..., \pi_m) > 0$.

We show by an inner induction that for $1 \leq j \leq m$,
\[
 \rho_\theta(x_1, ..., x_{j}, \pi_{j+1}, ..., \pi_m) \geq \rho_\theta(\pi_1,
...,
\pi_m).
\]
When $j = 1$ this holds, since $(\pi_2, ..., \pi_m)
\leq (\pi_1, ...,
\pi_{m-1})$ so that, by the  inductive assumption \[\rho_\theta(\pi_2, ...,
\pi_m)
\geq \rho_\theta(\pi_1, ..., \pi_{m-1}) \geq \rho_\theta(\pi_1, ..., \pi_m),\]
from which
\[
 \rho_\theta(x_1, \pi_2, ..., \pi_m)  \geq \rho_\theta(\pi_1, ..., \pi_m)
\]
follows by affine linearity.

Having shown \[\rho_\theta(x_1, ..., x_{j-1}, \pi_j, ..., \pi_m) \geq
\rho_\theta(\pi_1, ...,
\pi_m)\]
the case
\[
 \rho_\theta(x_1, ..., x_j, \pi_{j+1}, ..., \pi_m) \geq \rho_\theta(\pi_1, ...,
\pi_m)
\]
again follows by affine linearity from
\begin{align*}
 \rho_\theta(x_1, ..., x_{j-1}, 0, \pi_{j+1}, ..., \pi_m) &\geq
\rho_\theta(\pi_1, ...,
\pi_{j-1},0, \pi_{j+1}, ..., \pi_m) \\&\geq \rho_\theta(\pi_1, ...,
\pi_{m-1},0)\\&\geq
\rho_\theta(\pi_1, ..., \pi_m).
\end{align*}

\end{proof}

\bibliographystyle{plain}

\end{document}